\documentclass[a4paper,11pt]{amsart}
\usepackage{amsmath,amsfonts,amssymb,latexsym}
\usepackage{mathrsfs}
\usepackage{graphicx,float}
\usepackage[all,cmtip]{xy}
\usepackage{times}

\newtheorem{lemma}{Lemma}[section]
\newtheorem{corollary}[lemma]{Corollary}
\newtheorem{definition}[lemma]{Definition}
\newtheorem{proposition}[lemma]{Proposition}

\newtheorem{theorem}[lemma]{Theorem}
\newcommand{\field}[1]{\mathbb{#1}}
\newcommand{\scr}[1]{\mathscr{#1}}

\newcommand{\C}{\field{C}}
\newcommand{\F}{\field{F}}

\newcommand{\K}{\field{K}}

\newcommand{\R}{\field{R}}

\newcommand{\Rn}{\R^n}
\newcommand{\Rnbar}{\overline{\Rn}}

\newcommand{\Cal}[1]{\mathcal{#1}}
\newcommand{\cF}{\Cal{F}}
\newcommand{\ud}{{\rm{d}}}
\newcommand{\bbl}{{\bf bl}}
\newcommand{\bB}{{\bf B}}
\newcommand{\bG}{{\bf G}}
\newcommand{\bN}{{\bf N}}
\newcommand{\bbo}{{\bf 0}}
\newcommand{\bp}{{\bf p}}
\newcommand{\bP}{{\bf P}}
\newcommand{\br}{{\bf r}}

\newcommand{\bS}{{\bf S}}
\newcommand{\bt}{{\bf t}}
\newcommand{\bu}{{\bf u}}
\newcommand{\bv}{{\bf v}}

\newcommand{\bx}{{\bf x}}
\newcommand{\by}{{\bf y}}

\newcommand{\cH}{\Cal{H}}
\newcommand{\cM}{\Cal{M}}

\newcommand{\cU}{\Cal{U}}

\newcommand{\clos}{{\rm clos}}
\newcommand{\crit}{{\rm crit}}
\newcommand{\dd}{{\partial}}
\newcommand{\dt}{{\partial_t}}
\newcommand{\dbx}{{\partial_\bx}}
\newcommand{\gm}{\gamma}
\newcommand{\Gm}{\Gamma}
\newcommand{\kp}{{\kappa}}
\newcommand{\la}{{\langle}}
\newcommand{\Mbar}{\overline{M}}
\newcommand{\nbt}{{\nabla \bt}}
\newcommand{\nf}{{\nabla f}}
\newcommand{\nmbx}{{\nu_M^\bx}}
\newcommand{\nmt}{{\nu_M^t}}
\newcommand{\nF}{{\nabla F}}

\newcommand{\nubtm}{{\nu_{\,\bt_M}}}

\newcommand{\ra}{{\rangle}}

\newcommand{\scrX}{\scr{X}}
\newcommand{\scrV}{\scr{V}}

\newcommand{\Sgm}{\Sigma}

\newcommand{\tht}{{\theta}}

\newcommand{\vol}{{\rm{vol}}}
\newcommand{\ve}{\varepsilon}
\newcommand{\vp}{\varphi}
\numberwithin{equation}{section}

\begin{document}
\title[Gauss-Kronecker Curvature and equisingularity]{Gauss-Kronecker Curvature and equisingularity at infinity of definable families}

\author[N. Dutertre]{Nicolas Dutertre}
\author[V. Grandjean]{Vincent Grandjean}
%
\address{Laboratoire angevin de recherche en math\'ematiques, LAREMA, UMR6093, CNRS, UNIV. Angers, SFR MathStic, 2 Bd Lavoisier 49045 Angers Cedex 01, France.}
\email{nicolas.dutertre@univ-angers.fr}
\address{Departamento de Matem\'atica, Universidade Federal do Cear\'a
(UFC), Campus do Pici, Bloco 914, Cep. 60455-760. Fortaleza-Ce,
Brasil}
\email{vgrandjean@mat.ufc.br}
\thanks{{$ $ \\Both authors were partially supported by the ANR project LISA 17-CE400023-01}\\
Vincent Grandjean was supported by CNPq-Brazil grant 150555/2011-3 and FUNCAP/CAPES/CNPq-Brazil grant 305614/2015-0}
%
\date{\today}
%
%
\begin{abstract}
{Assume given a polynomially bounded o-minimal structure expanding the real numbers. 
Let $(T_s)_{s\in \R}$ be a globally definable family of $C^2$-hypersurfaces of $\R^n$. Upon defining 
the notion of generalized critical value for such a family, we show that
the functions $s \to |K(s)|$ and $s\to K(s)$, respectively the total absolute Gauss-Kronecker
and total Gauss-Kronecker curvature of $T_s$, are continuous in any neighbourhood of
any value which is not generalized critical. In particular this provides a necessary criterion 
of equisingularity for the family of the levels of a real polynomial.}
\end{abstract}
\maketitle
%
%
%
%
%

%
%
%
%
%
%
%
%
%
%
%
%
%
%
%
%
%
%
%
%
%
%
%
\section{Introduction}

One of the main goal of equisingularity theory (of families of subsets, functions, mappings) 
is to find relations between numerical data and regularity conditions. 
In the local complex analytic case, this subject has been widely studied since the end of the 60's and many interesting 
results, some of them now classical, have been established. For example, Hironaka \cite{Hironaka70} proved that the 
multiplicity is constant along the strata of a Whitney stratification of a complex analytic set. In \cite{TeissierCargese73} 
Teissier proved that a $\mu^*$-constant family of hypersurfaces with isolated singularities is Whitney equisingular. 
The reverse implication was proved later by Brian\c con and Speder \cite{BrianconSpederFourier76}. These results were 
extended to the case of {\em ICIS} by Gaffney \cite{GaffneyInvMath96}. Maybe the most important result of local 
complex analytic equisingularity theory is Teissier's polar equimultiplicity theorem \cite{TeissierLaRabida}, which 
states that Whitney regularity is equivalent to constancy of polar multiplicities. Teissier's results were refined and 
extended by Gaffney  \cite{GaffneyTopology93} to obtain sufficient conditions for equisingularity of a family of mappings.

When one considers global equisingularity problems, the first natural family to study is the family of fibres of a 
polynomial mapping. Following \cite{Th}, a polynomial function 
from $\K^n$ to $\K$, for $\K = \R$ or $\C$, is a smooth locally 
trivial fibration above the connected components of the complement of a (minimal) finite subset $B(f)$ of $\K$, 
called the set of bifurcation values of $f$. In the complex plane case, H\`a and L\^e \cite{HaLe} gave the following 
numerical criterion to characterize bifurcation values: \em A value $c$ does not lie in $B(f)$ if and only if the Euler 
characteristic of the fibres of $f$ is constant in a neighborhood of $c$.  \em This result was generalized by Parusi\'nski 
\cite{Par} to the case of complex polynomials with isolated singularities at infinity, and then by Siersma and Tib\u{a}r 
\cite{SiTi1} to the case of complex polynomials with isolated $\mathcal{W}$-singularities at infinity. In \cite{Tib1} 
Tib\u{a}r studies the more general situation of a $1$-parameter family of complex hypersurfaces, and proves a global 
version of the results of Teissier and Brian\c con and Speder mentioned above: Considering a family of complex affine 
hypersurfaces $X_\tau = \{ x \in \mathbb{C}^n \ : \ F(\tau,x)=0 \}$ given by a polynomial function 
$F : \mathbb{C} \times \mathbb{C}^n \mapsto \mathbb{C}$, he defines the notion of $t$-equisingularity at infinity and 
proves, under some additional conditions, that $t$-equisingularity at infinity is controlled by the constancy of a finite 
sequence of numbers, called the generic polar intersection multiplicities. As a consequence, if the family consists 
of non-singular affine hypersurfaces, then the constancy of the generic polar intersection multiplicities at $\tau_0$ 
implies that the family is $C^\infty$ trivial at $\tau_0$.

In the real semi-algebraic/sub-analytic setting (or more generally in the definable setting), it is hopeless to expect 
that \em constancy of numerical data is equivalent to regularity conditions. \em First, because of 
lack of connectivity, one cannot define invariants like the $\mu^*$-sequence, polar multiplicities or generic polar 
intersection multiplicities. However, using arguments from differential topology and integral geometry, one sees that 
these invariants admit geometric characterizations that still make sense in the real case. For instance, the multiplicity 
of a complex analytic germ is equal to its density \cite{DraperMathAnn69} and the $\mu^*$-sequence, the polar multiplicities 
and the generic polar intersection multiplicities are related to curvature integrals 
(see \cite{LangevinCommentarii79,LoeserCommentarii84,DutertreGeoDedicata2004,SiTi2}). 
Unfortunately, in the 
real situation, these geometric quantities do not belong to discrete sets and therefore, one cannot expect results 
relating their constancy to regularity conditions. It is more reasonable to study properties like continuity or Lipschitz 
continuity in the parameters of the family. The first result in this direction is due to Comte \cite{ComteAnnENS2000}, who 
established a real version of Hironaka's theorem, proving that the density is continuous along the strata of a 
$(w)$-stratification of a sub-analytic set. This result was generalized and strengthened by Valette \cite{ValetteAnnPol2008}: 
continuity of the density holds for 
$(b)$-regular stratifications and the density is Lipschitz continuous along the strata of $(w)$-stratifications.
Later Comte and Merle \cite{ComteMerleAnnENS2008} established a real version of Teissier's theorem \cite{TeissierLaRabida}. 
Using tools from integral geometry and geometric measure theory, they associated with each sub-analytic germ a sequence 
of numbers, called the local Lipschitz-Killing invariants, and showed that they are continuous along the strata of a 
$(w)$-stratification of a sub-analytic set. Recently, Nguyen and Valette \cite{NguyenValette} extended this continuity 
result to $(b)$-stratifications and moreover proved that these invariants are Lipschitz continuous along the strata of 
a $(w)$-stratification (see also the first author work \cite{DutertreAdvGeometry2018} for relations 
with the densities of polar images).

In the global real context, it is still true that the bifurcation set of a definable function from $\mathbb{R}^n$ to 
$\mathbb{R}$ is a finite set of points (see \cite{NeZa,LoZa,Tib2,Dac1}). In 
\cite{TiZa} Tib\u{a}r and Zaharia provided necessary and sufficient conditions for a real plane 
polynomial function to be locally trivial over the neighborhood of a regular value (see \cite{JoitaTibarProcRoySocEdim2017} 
for a generalization to a family of real curves). Unlike the complex case, their criterion is not only numerical 
but involves topological conditions at infinity. Later in \cite{CostedelePuenteJPAA2001}, Coste and de la Puente proved 
an equivalent version of Tib\u{a}r-Zaharia's results in terms of polar curves. Due to the links between polar curves and 
the Gauss-Kronecker curvature of the levels of a function provided by exchange formulas, it seems natural to study the 
variations of the total curvature of the levels (i.e. the integral of the Gauss-Kronecker curvature on the level) of a 
definable function, and to seek how bifurcation values interfere in these variations.

That is what the second author did in two papers. In \cite{Gra1} he considers a globally definable function 
$f : \mathbb{R}^n \mapsto \mathbb{R}$ of class at least $C^2$, and proved that the following functions:
$$ 
t \mapsto \int_{f^{-1}(t)} \kappa \, ,\;\;\hbox{ and } t \mapsto \int_{f^{-1}(t)} \vert \kappa \vert
$$
where $\kappa$ is the Gauss-Kronecker curvature, admit at most finitely many discontinuities. In \cite{Gra2} 
he proved that if the function $t \mapsto \int_{f^{-1}(t)} \vert \kappa \vert$ is continuous at a regular value $c$ which  
satisfies an extra condition, then $c$ is not a bifurcation value of $f$. He explained that for a real polynomial function 
with isolated singularities at infinity this extra condition is always satisfied, so this result 
can be interpreted as a real version of Parusi\'nski's result mentioned above. 

The aim of the present paper is to provide a kind of reverse implication of the latter mentioned result. 

\bigskip
We will work in the more general situation of a one parameter family of hypersurfaces. 
More precisely, we consider a globally definable function over an a priori given 
polynomially bounded $o$-minimal structure $F : \mathbb{R}^n \times \mathbb{R} \mapsto \mathbb{R}$ of class 
$C^{2+m}$ with non-negative integer $m$. Assuming that $0$ is a regular value of $F$, the $0$-level $M=F^{-1}(0)$ is thus 
a globally definable hypersurface in $\mathbb{R}^{n+1}$ of class $C^{2+m}$. We use the coordinates 
$(\bx,t)$ in $\mathbb{R}^n \times \mathbb{R}$ and 
we write $\bt_M : M \mapsto \mathbb{R}$, $(\bx,t) \mapsto t$ for the projection on the $t$-axis. 

For a value $c$ in $\R$, let $M_c=\bt_M^{-1}(c)$ and $T_c =\pi_M (M_c) \subset \mathbb{R}^n$, where $\pi_M$ is the 
projection from $M$ to $\mathbb{R}^n$. If $c$ is a regular value, then the hypersurface $T_c$ is oriented by $\partial_\bx F(\bx,c)$. 
Therefore, we consider the Gauss-Kronecker curvature $\kappa_c$ of $T_c$ and define two functions:
$$
c \mapsto K(c)=\int_{T_c} \kappa_c (\bp)d\bp \, , \; \; \hbox{ and } \;
c \mapsto \vert K \vert(c)=\int_{T_c} \vert  \kappa_c \vert (\bp) d\bp .
$$
By a straightforward adaptation of the methods of \cite{Gra1}, we show that these two functions have finitely many 
discontinuities (Theorem \ref{thm:continuity-curvatures}) and in Theorem \ref{thm:main}, we give a criterion on the regular value $c$ of $\bt_M$ for the 
function $t \mapsto \vert K \vert (t)$ to be continuous at $c$. Namely, we prove 

\smallskip\noindent
{\bf Theorem \ref{thm:main}}. \em
Let $c$ be a regular value taken by $\bt_M$ at which it is horizontally spherical at infinity. Then the total absolute curvature function 
$t \mapsto |K|(t)$ is continuous at $c$. Consequently the total curvature function $t \mapsto K(t)$ is continuous at $c$.
\em

\smallskip\noindent
The notion of  {\em horizontally sphericalness at infinity} is a regularity condition at infinity:
A regular value $c$ of ${\bf t}_M$ is horizontally spherical at infinity 
if for any sequence $({\bf p}_k)_{k \in \mathbb{N}}$ of $M$ converging at infinity to $({\bf u},c)$, ${\bf u}$ is orthogonal 
to the limit of the unitary gradients $\frac{\nabla {\bf t}_M}{\vert \nabla {\bf t}_M \vert} ({\bf p}_k)$.
A key ingredient of the proof of our main result, Theorem \ref{thm:main} is Lemma \ref{lem:dim-n-2} stating, informally, that
\em under these hypotheses there no accumulation of curvature at infinity nearby the level $c$. \em

We also prove that $(t)$-equisingularity at infinity implies horizontal sphericalness (Corollary \ref{cor:t-reg}). 
Therefore Theorem \ref{thm:main} shows that $(t)$-equisingularity at infinity implies continuity of the function $t \mapsto \vert K \vert (t)$. 
This can be considered as a first step towards a real version of Tib\u{a}r's result \cite{Tib1} mentioned above.

To be complete, we show here more than Theorem \ref{thm:main}. Its conclusion also holds true in any connected component of the pencil
of levels over a small interval of regular values $]c-\ve,c+\ve[$ (see Theorem \ref{thm:main-connected}). In other words
the connected components of the pencil of levels cannot compensate altogether the a priori possible discontinuities of some.

\medskip
The paper is organized as follows. Section \ref{section:misc} contains material on compactifications, $o$-minimal structures and Thom's 
$(a_f)$ condition. In Section \ref{section:t-reg}, we recall some facts about conormal geometry so that
we can introduce the notion of $t$-equisingularity. Sections  \ref{section:regularity}, \ref{section:comparing} and \ref{section:more} 
contain definitions and new results on regularity at infinity of globally definable $C^{2+m}$-families of hypersurfaces. 
In Section \ref{section:curvature}, we generalize the results of \cite{Gra1} to our situation. Section \ref{section:main} 
contains the proof of the main result. Section 9 deals with the particular case of the levels of a function. 

\medskip
\noindent
{\bf Acknowledgments.} The authors are very grateful to Si Tiep Dinh for useful, fruitful and inspiring conversations.
The second author would like to thank the I2M and  LAREMA for their working conditions while visiting the first author.
%
%
%
%
%
%
%
%
%
%
%
%
%
%
%
%
%
%
%
%
%
%
%
%
%
%
%
\section{Miscellaneous material}\label{section:misc}
Let $\Rn$ be the Euclidean space of positive dimension $n$.

Let $\la -, - \ra$ be the associated scalar product.
For any point $\bx$ of $\Rn$, let $|\bx|$ be the norm $\sqrt{\la \bx,\bx\ra}$ of $\bx$.

Let $\bS_R^{q-1}$  be the Euclidean sphere of $\R^q$ of center the origin and positive radius $R$.

Let $\bB_R^q$  be the closed Euclidean ball of $\R^q$ of center the origin and positive radius $R$.
When $q$ is understood we will only write $\bB_R$.

Let $\clos(-)$ denote the operation "taking the closure of" in $\Rn$. 
Each Euclidean space $\Rn$ embeds semi-algebraically in the closed unit-ball $\bB_1^n$, as its interior
via the mapping $\bx \mapsto \frac{\bx}{\sqrt{1+|\bx|^2}}$. We may then speak of $\bB_1^n$ as the \em spherical 
compactification of $\Rn$ \em (see the next section).

\bigskip
Let $\cM$ be an o-minimal structure expanding the real field $\R$. Assume it is polynomially bounded
and let $\F_\cM$ be the field of its exponents (\cite{vdDM,vdD}).
Any subset of any $\R^p$ definable in $\cM$ will be called below \em definable. \em

Usually globally definable subsets of $\cM$ are defined as definable subsets of any closed unit Euclidean ball.
 
Since each $\Rn$ embeds semi-algebraically in the closed unit-ball $\bB_1^n$, 
a subset $X$ of $\Rn$ is \em globally definable \em if it is definable
in the spherical compactification of $\Rn$.

Let $X$ be a subset of $\Rn$. A mapping $f:X\mapsto\R^p$ is globally definable if its graph is globally definable
in $\R^{n+p}$.

We would like to remind the following fact (see \cite{Dac1}): \em Let $\gm : [1,+\infty[\mapsto \R^n$ be a $C^1$ globally definable 
arc such that $\gm(t) \mapsto \infty$ as $t$ goes to $+\infty$. Then there exists a unit vector $\bu$ of $\bS^{n-1}$ 
such that
$$
\lim_\infty \frac{\gm}{|\gm|} = \bu = \lim_\infty \frac{\gm'}{|\gm'|} .
$$
\em

Let $f : (\R_{\geq 1},+\infty) \to \R$ be the germ at $\infty$ of a continuous globally definable function.
We write $f \sim t^e$ for an exponent $e$ in $\F_\cM\cup\{-\infty\}$, with the convention that $t^{-\infty} =0$ for large $t$, to mean
$$
f \sim t^e \Longleftrightarrow \lim_{t\to +\infty} \frac{f(t)}{t^e} \in \R^* \,.
$$
Note that there always exists such an exponent $e$.

\bigskip
Let $\bG(p,n)$ be the Grassmann manifold of $p$-vector subspaces of $\Rn$. 
We denote $\bG^\vee(p,n)$ the space of $p$-vector subspaces of the space $L(\Rn,\R)$ of linear forms over $\Rn$, and
we will call it sometimes the \em dual \em of $\bG(n-p,n)$.

\bigskip
We recall \em Thom's condition \em (or relative Whitney's condition $(a)$).

Let $X,Y$ be two connected $C^1$ submanifolds of a definable compactification of $\Rn$, such that 
$Y$ is contained in $\clos(X)\setminus X$. Let $g:(X\sqcup Y) \mapsto \R$ be a $C^1$ mapping, for
$X\sqcup Y$ the disjoint union of $X$ and $Y$. Let $\by$ be a point of $Y$. 

The function $g$ satisfies \em Thom $(a_g)$-condition at $\by$ \em if the following two conditions hold:
\\
(i) For any sequence $(\bx_k)_k$ of points of $X$ converging to $\by$ such that the sequence $(T_{\bx_k} X)_k$ converges to $T$
in the appropriate Grassmann bundle, then $T_\by Y$ is contained in $T$;
\\
(ii) For any sequence $(\bx_k)_k$ of points of $X$ converging to $\by$, such that the sequence $(T_{\bx_k} X)_k$ converges to $T$
which contains $T_\by Y$ and the sequence $(\ker \ud_{\bx_k} g)_k$ converges to 
$K$ in the appropriate Grassmann bundle, then $\ker \ud_\by g$ is contained in $K$.

\smallskip
In practice we want to stratify $g$ with Thom's condition asking that the strata $Y$ is contained in some specified
level of $g$.
%
%
%
%
%
%
%
%
%
%
%
%
%
%
%
%
%
%
%
%
%
%
%
%
%
%
%
%
%
\section{compactification and conormal geometry and $t$-equisingularity at infinity}\label{section:t-reg}
Let $\bbo$ be the origin of $\Rn$. 

As already seen in the previous section, we can compactify $\Rn$ as the closed unit ball $\bB_1^n$.

\medskip
An alternative presentation to the spherical compactification is
the spherical blowing-up $\bbl_\infty$ of $\Rn$ at infinity, that is the mapping given by 
$$
\begin{array}{rccl}
\bbl_\infty \; : & \bS^{n-1} \times ]0,+\infty[ & \mapsto & \Rn \setminus \{\bbo\} \\
 & (\bu,r) & \mapsto  & \frac{\bu}{r}
\end{array}
$$
It is a Nash diffeomorphism and a re-parametrization of $\Rn\setminus\{\bbo\}$ embedded in $\bB_1^n$.
It is more convenient to look at it this way since it is a good real avatar of the projective compactification 
$\bP\Rn$ (which in our globally definable context is not as relevant as in the algebraic case).
\\
We denote by $\bS_\infty^{n-1} := \bS^{n-1}\times 0$ the sphere at infinity. Let us denote and identify
$$
\Rnbar :=\Rn\sqcup \bS_\infty^{n-1} = (\bS^{n-1} \times [0,+\infty[)\sqcup \bbo,
$$
the spherical compactification of $\Rn$ at infinity, with boundary $\partial\Rnbar := \bS_\infty^{n-1}$, the sphere at infinity.

\medskip
Let $\overline{Z}$ be the closure of the subset $Z$ of $\Rn$ taken into $\Rnbar$.
The \em tangent link of $Z$ at $\infty$ \em is defined as
$$
Z^\infty := \overline{Z} \cap \bS_\infty^{n-1}.
$$
The \em tangent cone of $Z$ at infinity $C_\infty (Z)$ \em is defined as the (non-negative) cone
over $Z^\infty$. Whence $Z^\infty$ is not empty (equivalently $Z$ is not bounded) we also observe that
$$
Z^\infty :=\clos\left\{\bu\in\bS^{n-1} \,:\, \exists Z \, \ni (x_k)_k \to \infty  \hbox{ such that } \frac{x_k}{|x_k|} \to \bu\right\}.
$$

For any definable subset $Z$ of $\Rn$, it is globally definable if, by definition, $\overline{Z}$ is definable 
in $\Rnbar$. Thus whenever a subset $Z$ of $\Rn$ is globally definable, \em its tangent link at 
infinity $Z^\infty$ \em is definable and of dimension at most $\dim Z - 1$.

\medskip
Although heavy to define it is convenient to use the formalism of conormal geometry. We are especially interested in 
conormal geometry at infinity. 

Now let $Z=\{ (\bx,t) \in \Rn \times \R \, : \, G(\bx,t)=0 \}$, where $G : \Rn \times \R \mapsto \R$ is a 
globally definable function of class at least $C^2$, and let $\overline{Z}$ be its closure in $\overline{\Rn} \times \R$. 

We assume that $0$ is a regular value of $G$ and we consider $Z$ as a definable family $\{Z_t\}_{t \in \R}$ of 
hypersurfaces in $\Rn$. Let $g : Z \to \R$ be a definable function which we assume to be $C^1$. For any regular 
point $(\bx,t)$ of the function $g$, let $T_{(\bx,t)} g $ be the subspace of $T_{(\bx,t)} Z$ tangent at $(\bx,t)$ 
to the level of $g$ through $(\bx,t)$. Let us define the following subset of $\overline{\Rn} \times \R \times \bG^\vee(1,n+1)$:
$$
\scrX_g^\vee :=\clos\{(\bx,t,\xi) \in Z\setminus  \crit(g) \times \bG^\vee(1,n+1) \, : \, 
\xi (T_{(\bx,t)} g) = 0\},
$$
where $\bG^\vee(1,n+1)$ is the dual of $\bG(1,n+1)$.

\begin{definition}
The \em relative conormal space of $g$ \em is the space $\scrX_g^\vee$.
\end{definition}

Let $\pi_{n+1} : \overline{\Rn} \times \R \times \bG^\vee (1,n+1) \mapsto \overline{\Rn} \times \R$ 
be the projection given as $\pi_{n+1}(\bx,t,\xi)=(\bx,t)$. 
\begin{definition}
The \em relative conormal space of $g$ at infinity \em is the space $\scrX_g^\infty$ defined as 
$$
\scrX_g^\infty := \pi_{n+1}^{-1}(Z^\infty) \cap \scrX_g^\vee,
$$
where $Z^\infty =\overline{Z} \cap ( \bS_\infty^{n-1} \times \R)$.

For any $\bp \in Z^\infty$, let $(\scrX_g^\infty)_\bp$ be the fibre of $\scrX_g^\vee$ above $p$, that is 
$(\scrX_g^\infty)_\bp = \pi_{n+1}^{-1}(\bp) \cap \scrX_g^\vee$.
\end{definition}

\medskip
We introduce now the notion of $t$-equisingularity \cite{Tib1} adapted to the context of Section \ref{section:regularity}.

\smallskip

Let $\br_Z:Z \mapsto\R$ be defined as $(\bx,t) \mapsto |\bx|$. It is continuous globally definable and
$C^1$ outside $0\times\R \cap Z$. 

The \em space of characteristic covectors $\scr{C}$ of $Z$ at infinity \em is the subset of 
$\Rnbar\times \Rn \times \bG^\vee(1,n+1)$ defined as 
$$
\scr{C}(Z) := \scrX_{\br_Z}^\infty.
$$
It is closed and definable.

Let $\tau:\Rnbar\times\R$ be defined as $(\bx,t) \mapsto t$. 

The following notion is due to Tib\u{a}r \cite{Tib1}:
\begin{definition}\label{def:t-reg}
Let $\bp $ in $Z^\infty$. 

(i) The family $\{Z_t\}_{t \in \R}$ is $t$-equisingular at $\bp$ if 
$$
\scr{C}(Z)_\bp \cap (\scrX_\tau^\infty)_\bp = \emptyset.
$$

(ii) The family $\{Z_t\}_{t \in \R}$ is $t$-equisingular at infinity at $c$ if it is
$t$-equisingular at $\bp$ for all $\bp $ in $Z^\infty \cap \tau^{-1}(c)$.
\end{definition}
The definition above is slightly different from those given in \cite{SiTi1,Tib1,DRT}, since 
there it is given via the projective compactification of $\Rn$. Anyhow they are equivalent.

\smallskip
Any co-vector in $(\scrX_\tau^\infty)_\bp$ has kernel the horizontal 
hyperplane $\Rn\times 0$. We deduce that any limit of tangent spaces $T$ of $\br_Z$ at 
$\bp$ does not lie in $\Rn\times 0$ whenever $\{Z_t\}_{t \in \R}$ is $t$-equisingular at $\bp$.
%
%
%
%
%
%
%
%
%
%
%
%
%
%
%
%
%
%
%
%
%
%
%
%
%
%
%
%
%
\section{Regularities at infinity for definable families of hypersurfaces}\label{section:regularity}
We present here two regularity conditions at infinity for the function restriction of a coordinate projection 
along a globally definable one parameter family of hypersurfaces. In the next section, we will 
compare altogether these regularity conditions with $t$-equisingularity, introduced in the previous section.

\bigskip
Let $F:\R^n_\bx\times \R_t \mapsto \R$ be a $C^{2+m}$ globally definable function, for some non-negative integer $m$.

Assuming that $\Rn\times\R$ is equipped with the canonical Euclidean structure, 
let $\nF$ be the gradient field of $F$. Without further hypotheses, the real number $0$ may be a critical value of $F$, 
and $\nF$ may be vanishing on the zero level of $F$.

\medskip\noindent
{\bf Working Hypothesis:} \em $0$ is a regular value of $F$. \em

\smallskip
Let $M$ be the zero locus $F^{-1}(0)$ of the function $F$, which is a closed globally definable subset of $\R^{n+1}$
and a $C^{2+m}$ hypersurface. Let $\Mbar$ be its closure in $\Rnbar\times\R$.

We define two mappings $(\pi,\bt):\Rnbar\times\R \mapsto \Rnbar\times\R$ obtained respectively as the restrictions
of the projection over $\Rnbar$ and over $\R$, and both are semi-algebraic.

Let $\bt_M$ be $\bt|_M$ the restriction of $\bt$ to $M$ and let $\pi_M$ be the restriction of $\pi$ to $M$, both are
$C^{2+m}$ and globally definable mappings. 
Let us write $M_c := \bt_M^{-1}(c)$ and $T_c := \pi_M(M_c)$ subset of $\Rn$. 
\begin{definition}\label{def:trivial}
Let $c$ be a value taken by $\bt_M$. The function $\bt_M$ is said locally $C^k$ trivial at $c$ if 
there exists a positive real number $\ve$ such that $\bt_M^{-1}(]c-\ve,c+\ve[)$ is a trivial $C^k$-bundle 
with fibre $M_c$.
\end{definition}
Mimicking what was done for level hypersurfaces of functions \cite{LoZa,TiZa,Dac1,DaGr1,DaGr2,Gra2}, sufficient conditions 
about the gradient of $\bt_M$ guarantee trivialization (see below).
Since $M$ is globally definable and each of its connected component is orientable, let $\nu_M$ be a $C^{1+m}$ globally
definable unitary field normal to $M$. Since $0$ is not a critical value of $F$, we choose 
$$
\nu_M := \frac{\nF}{|\nF|} = \nmbx + \nmt \dt
$$ 
where $\nu_M^\bx$ is the component of $\nu_M$ in $\Rn \times 0$, 
and writing 
$\nF = \dbx F + \dt F\dt$, where $\dbx F$ lies in $\Rn\times 0$.

Let $\bp = (\bx,t)$ be a point of $M$. We have
$$
T_\bp M = \{(\bu,w) \in \Rn\times \R \, : \, \la \dbx F, \bu \ra + \dt F\cdot w = 0\}.
$$
It is easy to prove the following relation:
$$
\nbt_M = - \,\frac{\dt F}{|\nF|^2} \dbx F +  \frac{|\dbx F|^2}{|\nF|^2} \dt = -  \nmt\nmbx + |\nmbx|^2\dt
$$
and thus 
$$
|\nbt_M| = \frac{|\dbx F|}{|\nF|} = |\nmbx|.
$$
The critical locus of $\bt_M$ is 
$$
\crit(\bt_M) = \{\bp \in M \, : \, \dbx F(\bp) = 0\}.
$$
Since $M$ is a $C^{2+m}$ orientable hypersurface, the function $\bt_M$ is $C^{2+m}$ as well.
Since it is globally definable, the set of its critical values $K_0(\bt_M) := \bt_M(\crit(\bt_M))$ 
is finite. 

Let $\nubtm : M\setminus \crit(\bt_M)\mapsto \bS^n$ be the unitary gradient of $\nbt_M$,
$$
\nubtm := \frac{\nbt_M}{|\nbt_M|}.
$$

\medskip
The Local Conical Structure Theorem ensures the existence of a positive number $S_M$ such that 
for any $S > S_M$ the hypersurface $M$ is transverse with $\bS_S^n$, the Euclidean sphere of radius $S$.
As a consequence of this fact we also have:
\begin{lemma}\label{lem:transverse-horizontal-sphere}
For any $A > max_{c\in K_0(\bt_M)}|c|$, there exists $R_A$ such that for any $R> R_A$
the globally definable $C^{2+m}$ hypersurface $M\cap (\Rn\times ]-A,A[)$ is transverse to the cylinder $\bS_R^{n-1}\times \R$.
\end{lemma}
\begin{proof}
Let $A \gg 1$ be given.
Let us define the following subset
$$
\Sgm := \{(\bx,t)\in M \, : \, \bx \wedge \nF(\bx,t) =0\}.
$$
Note that $\Sgm\setminus \bbo \times \R$ is contained in $M \cap\{\dt F = 0\}$
and that $\Sgm$ is a closed globally definable subset of $M$.

\smallskip
Let us assume that the statement of the lemma is not true. Thus there exists a $C^1$ globally definable 
path $\gm:[1,+\infty[ \mapsto \Sgm \cap \Rn\times ]-A,A[$ such that 
$\gm(s) \to (\bu,c)$ in $M^\infty$ as $s$ goes to $+\infty$, with $|c| < A$.
\\
We can parameterize $\gm$ in such a way that $|\gm (s)| = s$, which gives the following
$$
\gm(s) = (s\bu,0) + s^e \bv(s)
$$
for a $C^1$ and globally definable mapping $s\mapsto s^e\bv(s) \in \Rn\times\R$ such that $\lim_\infty \bv \neq (0,0)$ and
$e< 1$. We also have that $s\mapsto \nu_M (s) = \frac{\nF}{|\nF|}(\gm(s))$ goes to $\nu$ in $\bS^n$ as $s$ goes to $\infty$. 
Note that $\frac{\gm'(s)}{\vert \gm'(s) \vert}$ goes to $(\bu,0)$ as $s$ goes to $+\infty$. Since 
$$
\gm(s) \wedge \nu_M(s) = 0 \; \mbox{ and } \; \la \gm'(s), \nu_M(s) \ra = 0\, ,
$$
we deduce that
$$
\bu \wedge \nu = 0 \; \mbox{ and } \; \la \bu,\nu\ra = 0,
$$
which is absurd.
\end{proof}

\medskip
We can introduce now the Malgrange regularity condition at infinity.
\begin{definition}\label{def:malgrange-acv}
Let $c\in\R$ be a value. 

\smallskip\noindent
(i) The function $\bt_M$ satisfies the Malgrange condition at $c$ if 
there exist positive constants $R,\ve,A_c$ such that 
\begin{equation}\label{eq:malgrange}
|\bx| > R, \; |t - c| < \ve  \Longrightarrow |\bx|\cdot |\nbt_M(\bx,t)| \geq A_c
\end{equation}
which is equivalent to 
\begin{equation}\label{eq:malgrange-bis}
|\bx| > R, \; |t - c| < \ve  \Longrightarrow |\bx|\cdot |\dbx F| \geq A_c |\nF| . 
\end{equation}

\smallskip\noindent
(ii) A value $c$ which is not satisfying the Malgrange condition is called \em 
an asymptotic critical value \em (ACV for short). Let $K_\infty(\bt_M)$ be the set 
of ACV of $\bt_M$.
\end{definition}

\smallskip
Similarly to the case of real or complex polynomial families \cite{Par,Tib1,Tib2,TiZa} we find
\begin{theorem}[see also \cite{LoZa,Kur,Dac1,DaGr1}]
(i) There exists a finite subset $B(\bt_M)$ of $\R$ such that the function $\bt_M$ is a locally $C^{1+m}$ trivial
at any value $c$ not lying in $B(\bt_M)$.

\smallskip\noindent
(ii) $B(\bt_M) \subset K_0(\bt_M) \cup K_\infty(\bt_M)$.

\smallskip\noindent
(iii) $K_\infty(\bt_M)$ is finite.

\smallskip\noindent
(iv) If $c$ is a regular value taken by $\bt_M$ and does not lie in $K_\infty(\bt_M)$, 
the local trivialization can be realized by a vector field colinear to $\nubtm$.
\end{theorem}
\begin{proof}
We are going to sketch the proofs of (iii) following \cite{Dac1} and (iv) following \cite{Dac1,Dac2,DaGr1}.
Both (i) and (ii) can be deduced from these two points.

\smallskip
For simplicity we write $f$ for $\bt_M$.

\smallskip
Since $\cM$ is polynomially bounded, there exists a globally definable function $\rho :[1,+\infty[ \mapsto \R_+$ 
such that (see \cite[Lemma 3.3]{Dac1}):

\smallskip
\begin{center}
(1) $\lim_{R\to \infty} \, R^{-1} \rho(R) = +\infty\;\;\;$ and $\;\;\;\;$ (2) $K_\infty (f) = K_\infty^{\rho}(f)$
\end{center}
where 
$$
K_\infty^{\rho}(f) := \{c\in \R: \exists (\bx_k,t_k)\in M\, , \, \bx_k\to \infty \mbox{ and } t_k \to c \mbox{ such that }
\rho(|\bx_k|)\cdot|\nf(\bx_k,t_k)| \to 0\} .
$$
In particular for $R$ large enough there exists an exponent $\alpha$ of $(\F_\cM)_{>1}$ such that $\rho(R) \geq R^\alpha$.

\smallskip
Following the steps of \cite[Theorem 3.4]{Dac1}, we show that $K_\infty^\rho (f)$ is finite.

Assume that there exists $c'>c$ such that for each $c\leq t \leq c'$ the level $\{f=t\}$ is neither empty
nor is a critical level. Let us consider the following subset
$$
\Delta := \{(s,w) \in \R^2 \, : \, \forall \delta > 0 \, , \, \exists \ve > 0, \exists R > 0 \, : \,
|\bx| > R , |\bt_M(x) -w|< \ve \Longrightarrow |\bx|\cdot|\nbt_M (\bx)| - s|<\delta\}.
$$
This subset is globally definable. Let $\tht : [c,c'] \mapsto [0,+\infty[$ be the function defined as follows
$$
\tht (t) := \inf\{\Delta \cap \{w=t\}\}.
$$
It is globally definable. We wish to show that it vanishes only finitely many times on $[c,c']$.
Assume that $\theta$ is identically $0$ over $[c,c']$ (up to work with a smaller $c'$).
Under these hypotheses the globally definable subset 
$$
\Sgm := \{\rho(|\bx|)\cdot|\nf (\bx,t)| < f(\bx,t) -c\}
$$
is not empty outside of a compact subset of $M \cap \bt_M^{-1}[c-\ve,c+\ve]$ (see \cite[p. 40]{Dac1}). 
By definition of $\Sgm$, there exists a $C^1$ globally definable arc going to infinity 
$\gm :[1,+\infty[\mapsto \Sgm$  such that 
$$
\lim_\infty f\circ\gm = c.
$$
Let $h = f\circ\gm$, and let us parameterize $\gm$ such that $|\bx(\gm(R))| = R$, so that $|\gm'|$ goes to $1$ at infinity. We find
$$
0 < -h' \leq \frac{h}{|\rho|}|\gm'| < 2\frac{h}{|\rho|}.
$$
Let $R_0$ be large enough and let $u(R) = 2h(R_0) - h(R)$ once $R \geq R_0$, and let $a > 1$ be such that 
$$
\lim_\infty R^{-a}\rho(R) = +\infty.$$
We deduce for $R\geq R_0$ (up to taking a larger $R_0$) 
$$
0< u'(R)< 2\frac{u(R)}{|\rho(R)|} < \frac{u}{R^a}.
$$
Applying the Gronwall Lemma provides for $R\geq R_0$
$$
u(R) \leq h(R_0)\cdot \lambda(R_0) \; \mbox{ with } \; \lambda(R_0) := \exp\left(\frac{1}{(a-1)R_0^{a-1}}\right) .
$$
We know that $\lim_\infty u =  2h(R_0)$ but we can choose $R_0$ a priori such that $\lambda(R_0) <2$, concluding that 
the function $\theta$ cannot vanish identically over $[c,c']$.

\medskip
Point (iv) is of importance for the rest of the paper so we sketch its proof as a
 variation of the proof of \cite[Theorem 3.5]{DaGr1}. Let
$$
\chi := \frac{1}{|\nbt_M|^2}\nbt_M\, .
$$
Any trajectory $\gm$ of $\chi$ is parameterized by the levels of $\bt_M$: starting at a point of $M_c$ we find
$$
\bt_M(\gm(s)) = c + s \, .
$$
Since the Malgrange condition is not affected by a change of origin of $\Rn$, 
we can assume that for every small enough positive real number $\ve$ there exists a constant
$A_\ve$ so that 

\begin{equation}\label{eq:malgrange-global}
|\bx|\cdot |\nbt_M(\bp)| \geq A_\ve \, \mbox{ for all } \, \bp = (\bx,t) \in \bt_M^{-1}[c-\ve,c+\ve] \, .
\end{equation}
For $|s|\leq \ve$ we deduce 
$$
|\gm(s)| \leq |\gm(0)| + \int_0^{|s|} \frac{1}{|\nbt_M(\gm(z))|}dz  \, .
$$
Combining this latter inequality with Gronwall Lemma provides
$$
|\gm(s)| \leq |\gm(0| \cdot \exp\left(\frac{|s|}{A_\ve}\right) \,.
$$
Since $\chi$ is $C^{1+m}$ in $M\setminus \crit(\bt_M)$, the function $\bt_M$ is $C^{1+m}$-trivial at $c$
by the flow of $\chi$ with initial conditions along $M_c$. 
\end{proof}
\begin{definition}
The set \em of generalized critical values \em is defined as 
$$
K(\bt_M) := K_0(\bt_M) \cup K_\infty(\bt_M).
$$
\end{definition}
The Malgrange condition at a regular value $c$ encodes the geometry at infinity of the pencil of nearby fibres.
Indeed we have the following
\begin{lemma}\label{lem:orthogonal-malgrange}
Let $c$ be a value taken by $\bt_M$ which is not a generalized critical value. 
Let $(\bp_k)_k$ be a sequence of points of $M$ converging in $\Mbar$ to 
$(\bv,c)$ in $\bS_\infty^{n-1}\times \{c\}\cap \Mbar$ while $\bt_M(\bp_k)$ goes to $c$.
Assume that $\nubtm(\bp_k)$ converges in $\Mbar$ to $\nu$ in $\bS^n$. 
Then $\nu$ is orthogonal to $\bv$.
\end{lemma}
\begin{proof}
These limits can be achieved along a $C^1$ globally definable path $]0,1[ \mapsto M, \; s\mapsto \bp(s) = (\bx(s),t(s))$ as $s$
goes to $0$ with $\lim_0 \bp = (\bv,c) \in M^\infty$. We choose the parameterization of $\bp$ so that $|\bx(s)| = s^{-1}$.
Let $\bt(s) := \bt_M(\bp(s))$ and so on. Let us write
\begin{eqnarray*}
\bt (s) & = & c + A s^a + o(s^a) \\
\bx (s) & = & s^{-1}\bv + o(s^{-1}) \\ 
\nubtm (s)& = & \nu + s^d \nu_1 + o(s^d) \\
\nmbx(s) & = & s^e \nu^\bx + o(s^e) \\
\nmt (s) & = & s^f\nu^t + o(s^f)
\end{eqnarray*}
where $a,d,e,f \in (\F_\cM)_{\geq 0}\cup\{+\infty\}$ with $a,d$ positive exponents, $\min(e,f) =0$,
and $A\in \R$, 
$\nu_1\in\Rn\times\R$, $\nu^\bx \in \R^n \times 0$, $\nu^t \in \R$
are non-zero vectors whence the corresponding exponent is not $\infty$ and 
$$
\nu = -\nu^t\frac{\nu^\bx}{|\nu^\bx|} + |\nu^\bx|\dt.
$$

We deduce that there exists a continuous definable function $\vp : (\R_{\geq 0},0) \mapsto \R$ with $\vp(0) = Aa$, such that 

\smallskip
\begin{center}
$\bt'(s) = s^{a-1}\vp(s)\;\;$ and $\;\;\bx' (s)  =  -s^{-2}\bv + o(s^{-2})$.
\end{center}
Using the Malgrange condition provides
$$
1 \geq |\nmbx(s)| = |\nbt_M (s)| \geq A_c \,s 
$$
for some positive constant $A_c$. Thus $e\leq 1$. Since 
$$
0 = \la\nmbx + \nmt \dt,\bx' + \bt'\dt\ra = \la \nmbx, \bx' \ra + \nmt \bt'\;\;
$$
we deduce that
$$
\nmt \, \bt' = - \la \nmbx, \bx' \ra  = s^{e-2}[-\la \nu^\bx , \bv \ra  + o(1)].
$$
From this last equation we deduce that there exists an exponent $e'\geq e$ such that
$$
s^{a-1+f} \sim | \nmt \, \bt' | = | \la \nmbx, \bx' \ra | \sim s^{e'-2},
$$
so that $\la \nu^\bx , \bv \ra = 0.$
\end{proof}

\medskip
To conclude this section we  introduce a final regularity condition.  
\begin{definition}\label{def:spher-infty}
Let $c$ be a regular value of $\bt_M$ taken by $\bt_M$.

The function $\bt_M$ is \em horizontally spherical at $c$ at infinity \em
if for any sequence $(\bp_k)_k$ of $M$ converging to $(\bu,c) \in M^\infty$, then
\begin{equation}\label{eq:def-HorSph}
\left\langle \lim_\infty \frac{\nu_M^\bx}{|\nu_M^\bx|}\, , \, \bu \right\rangle = 0,
\end{equation}
where $\lim_\infty \frac{\nu_M^\bx}{|\nu_M^\bx|}$ means the closed set of all the possible accumulation values, as $k$ 
goes to infinity, of the unitary vector field of $\frac{\nu_M^\bx}{|\nu_M^\bx|}$ along the sequence $(\bp_k)_k$.
\end{definition}
Note that the following holds true:
\begin{lemma}
The condition of Equation \ref{eq:def-HorSph} is equivalent to 
$$
\langle \lim_\infty \nu_{\bt_M} , \bu\rangle = 0
$$
along any sequence $(\bp_k)_k$ of $M$ converging to a point 
$(\bu,c)$ in $M^\infty$.
\end{lemma}
Indeed, similarly to what has been done for globally definable functions, we have the following:
%
\begin{proposition}\label{prop:spherical-exponent}
Let $c$ be a regular value taken by $\bt_M$. 
The function $\bt_M$ is horizontally spherical at $c$ at infinity if and only if there 
exists an exponent $e_c$ in $\F_\cM \cap ]-\infty,1[$ 
and a positive constant $E_c$, such 
that there exist positive real numbers $\ve$ and $R$ such that 
\begin{equation}\label{eq:bochnak-Lojasiewicz-exponent}
(\bx,t) \in \bt_M^{-1}([c-\ve,c+\ve])\setminus \bB_R \Longrightarrow 
|\bx|\cdot|\nbt_M| \geq  E_c |\bt_M(\bx,t) - c|^{e_c}.
\end{equation}
\end{proposition}
\begin{proof}
In this globally definable and polynomially bounded 
context, we can show (as in \cite{DaGr2}) that a Bochnak-\L ojasiewicz inequality type at the value $c$ not in $K_0(\bt_M)$
at infinity holds: there exists a positive constant $L_c$ such that there exist positive real numbers $\ve$ and $R$ such that 
\begin{equation}\label{eq:bochnak-Lojasiewicz}
(\bx,t) \in \bt_M^{-1}([c-\ve,c+\ve])\setminus \bB_R \Longrightarrow |\bx|\cdot|\nbt_M| \geq L_c|\bt_M(\bx,t) - c|.
\end{equation}

\medskip\noindent
1) Assume $\bt_M$ is horizontally spherical at $c$ at infinity. 

\smallskip\noindent
Let $\bp: ]0,1[\mapsto M$ be any continuous globally definable path such that it goes to $(\bu,c)$
as $s$ goes to $0$. Writing
$\bp = (\bx,\bt)$ and parameterizing as $|\bx(s)| = s^{-1}$, we have 
$$
\bp(s) = (s^{-1} \bu + o(s^{-1}), c + As^a + o(s^a))
$$
for $A\neq 0$ and $a\in (\F_\cM)_{>0} \cup \{+\infty\}$. The numbers $A$ and $a$ depend on the choice of the
path $s\mapsto \bp(s)$. We obtain that along $\bp$ there exists $a' \leq a$ such that
$$
|\bx|\cdot|\nbt_M | \sim s^{a'}.
$$
Note that 
$$
a' < a \; \Longleftrightarrow \; 
\lim_0 \left\langle \nu_{\bt_M}, \frac{\bp}{|\bp|} \right\rangle = 0.
$$
In particular the latter equivalence 
shows that
$$
\bt_M (\bx,t) \to c \mbox{ \rm as } \bx \to +\infty \; \Longrightarrow \; \frac{|\bt_M(\bx,t)) - c|}{|\bx|\cdot |\nbt_M(\bx,t)|} \to 0
\mbox{ \rm as } \bx \to +\infty \,.
$$

\bigskip
Let $\ve_0$ be a small enough positive number such that $[c-\ve_0,c+\ve_0]$ contains
only a single asymptotic critical value: $c$. Let $R_0$ be a positive large enough number.
Let $V_{\ve_0,R_0}$ be the globally definable subset defined as
$$
V_{\ve_0,R_0} := \{(\bx,t)  \in M \, : \, |\bt - c| \leq \ve_0\, , |\bx|\geq R_0\}.
$$
For $R\geq R_0$ let $\mu_0: [R_0,+\infty[ \to \R$ be defined as
$$
\mu_0 (R) := \min \{|\bx|\cdot |\nbt_M(\bx,t)| \; \mbox{ \rm for } \; (\bx,t) \in V_{\ve_0,R_0} \hbox{ and }  |\bx| = R\} .
$$
The function $\mu_0$ is globally definable and tends to $0$ as $R$ goes to infinity
since $c$ is an ACV. If $R_0$ is large enough, we can write
$$
\mu_0 (R) = A_0 R^{-a_0}(1+o(1)) \; \mbox{\rm with } \; A_0> 0, \, a_0\in (\F_{\cM})_{>0} .
$$

Let $V_0$ be the closure of $V_{\ve_0,R_0}$ in $\overline{M}$, thus $V_0$ is compact in $\overline{\R^n}\times\R$.
Let $W_0$ be the part at infinity of $V_0$, that is
$$
W_0 : = V_0 \cap (\bS_\infty^{n-1} \times \R).
$$
The function 
$$
\psi_0 : V_0\setminus W_0 \ni (\bx,t) \to \frac{|\bt_M(\bx,t)) - c|}{|\bx|\cdot |\nbt_M(\bx,t)|}
$$
extends continuously and definably over $V_0$ taking the value $0$ along $W_0$, by hypothesis
of horizontal sphericalness. 
In the same way, the function
$$
\rho_0 : V_0 \setminus W_0 \ni (\bx,t) \to |\bx|^{-1}
$$
also extends continuously and definably over $V_0$ taking the value $0$ along $W_0$.
Furthermore we see that 
$$
\rho_0 = 0 \Longrightarrow \psi_0 = 0 \, .
$$
Thus by a \L ojasiewicz argument, there exist a positive exponent $b$ and a positive constant $B$ such that 
in $V_0$ the following inequality holds true:
$$
\psi_0 \leq  B \rho_0^b \Longleftrightarrow \psi_0 \leq B |\bx|^{-b}.
$$
Let $\mu_1$ be the function defined as follows:
$$
\mu_1 : V_0 \setminus W_0 \ni (\bx,t) \to \mu_0 (|\bx|).
$$
The function $\mu_1$ is globally definable, continuous and extends continuously to $V_0$ taking the value $0$ along $W_0$.
Therefore we deduce that in $V_0\setminus W_0$ we have
$$
|\bt_M(\bx,t) - c| \leq C_0 \cdot \mu_1^\frac{b}{a_0} \cdot|\bx| \cdot |\nbt_M(\bx,t)| \leq C_0\cdot(|\bx|\cdot |\nbt_M(\bx,t)|)^\frac{b+a_0}{a_0} 
$$
where $C_0$ is a positive constant.
This latter inequality provides the announced result.

\medskip\noindent
2) Assume the inequality holds.

\smallskip\noindent
Let $\bp :]0,1[\mapsto M$ be a globally definable continuous path such that $\lim_0 \bp = (\bu,c)$.
Writing $\bp = (\bx,\bt)$ and parameterizing as $|\bx(s)| = s^{-1}$, we have that
\begin{eqnarray*}
\bt(s) & = & c + As^a + o(s^a) \\
\bp(s) & = &(s^{-1}\bu + o(s^{-1}), \bt(s)) \in \Rn\times\R \\
\nu_M^\bx (s) & = & s^b \nu + o(s^b) \in \Rn\times 0 \\
\nu_M^t (s) & = &  s^d (\lambda v) + o(s^d) \in \Rn\times 0 \\
\end{eqnarray*}
with $A\neq 0$ and $a\in (\F_\cM)_{>  0} \cup \{\infty\}$ while $b,d\in (\F_\cM)_{\geq 0}$, $\min(b,d) = 0$, with 
$\lambda \neq 0$ and $\nu \in \Rn\setminus 0$.
\\
Since the path $\bp$ lies on $M$, we know that 
$$
\la \nu_M,\bp'\ra = \la \nu_M^\bx,\bx'\ra + \nu_M^t \bt' = 0
$$
from which we deduce
\begin{equation}\label{eq:bda}
b - 2 \leq d + a - 1 \,.
\end{equation}
We want to show that $\nu$ is orthogonal to $\bu$, in other words $b < d+a+1$. 

\smallskip\noindent
We have the following estimates
\begin{eqnarray*}
|\nbt_M|(s) = |\nu_M^\bx|(s) & \sim & s^b \\
|\bx|\cdot |\nbt_M|(s) & \sim & s^{b-1}.
\end{eqnarray*}
Using Inequality (\ref{eq:bochnak-Lojasiewicz-exponent}), we get 
\begin{equation}\label{eq:b<a+1}
b-1 \leq e_c \cdot a < a \hbox{ and } b< a+1.
\end{equation} 
Since $d$ is non negative, this  yields the orthogonality of $\bu$ and $\nu$.
\end{proof}
%
%
%
%
%
%
%
%
%
%
%
%
%
%
%
%
%
%
%
%
%
%
%
%
%
%
%
%
%
%
%
%
%
\section{Comparing regularity conditions and triviality}\label{section:comparing}
$ $ 

We are working within the context of Section \ref{section:regularity}.

We have introduced previously three regularity conditions at infinity for the 
function $\bt_M$. We are going to compare them here.

The hypersurface $M \subset \Rn \times \R$ is the definable family of the hypersurfaces $\{T_t\}_{t \in \R}$ 
of $\Rn$ and $\overline{M}$ is its closure in $\overline{\Rn} \times \R$. Let $M^\infty$ be the intersection of 
$\overline{M}$ with the boundary at infinity $\bS_\infty^{n-1} \times \R$. By Lemma \ref{lem:transverse-horizontal-sphere}, 
the globally definable function $\br_M : \Rn\times\R \mapsto \R$,  defined as $(\bx,t) \mapsto \vert \bx \vert $, is
transverse to $M\cap \Rn\times\R\times]-A,A[$ for some positive given
$A$ whenever $\bx$ is large enough. 

In Section \ref{section:t-reg} was defined
$$
\scr{C}(M):= \scrX_{\br_M}^\infty 
$$ 
the \em space of characteristic covectors of $M$ at infinity, \em 
which is a closed definable subset of $\bS_\infty^{n-1} \times \R \times \bG^\vee(1,n+1)$.

\medskip 
From Definition \ref{def:t-reg}, we also know that: (i) the family $\{T_t\}_{t \in \R}$ is $t$-equisingular at $\bp \in M^\infty$ if 
\begin{equation}\label{eq:t-reg-tm}
\scr{C}(M)_{\bp} \cap (\scrX_{\tau}^\infty)_{\bp} = \emptyset,
\end{equation}
where $\tau :\overline{\Rn} \times\R \mapsto \R$ is the projection on the last factor
and, (ii) the family $\{T_t\}_{t \in \R}$ is $t$-equisingular  at infinity at $c$ if it is
$t$-equisingular at $\bp \in M^\infty$ for all $\bp \in M^\infty \cap \tau^{-1}(c)$.

Let $\bp = (\bu,c) \in M^\infty$. The family $\{T_t\}_{t \in \R}$ is $t$-equisingular at $\bp$ 
if for any sequence $\bp_k = (\bu_k,t_k)$ converging to $\bp$ such that the
sequence of $T_k'$, the tangent space to the level of $\br_M$ through $\bp_k$, converges to $T'$, 
then the latter is not contained in $\Rn\times  0$. 
This definition is more geometric than the Malgrange condition, which is of interest since
we have the following:
\begin{proposition}[see \cite{DRT} for functions]\label{prop:t-reg-malg}
If the family $\{T_t\}_{t \in \R}$ is $t$-equisingular  at infinity at $c$ then
the function $\bt_M$ satisfies the Malgrange condition at $c$.
\end{proposition}
\begin{proof}
Suppose that the Malgrange condition is not satisfied at $c$. There is a globally definable path 
$\gm : [1,+\infty[ \mapsto M$, $\gm=(\gm_\bx,\gm_t)$, with $\lim_\infty \gm = \bp = (\bu,c)\in M^\infty$ and such
that 
$$
(|\bx|\cdot|\nbt_M|) \circ\gm \to 0 \; \mbox{ as } \, r \to +\infty \, .
$$
Equivalently 
$$
\left(|\bx|\cdot \left|\frac{\dd_\bx F}{\nabla F}\right| \right) \circ \gm \to 0 \,.
$$

The following vector $V$ in $\R^n \times \R $:
$$
V = \frac{\dd_\bx F}{\vert \nabla F \vert} \left( \begin{array}{c}
\bx \\  0 \end{array} \right) + \frac{1}{\vert \nabla F \vert} \left( \begin{array}{c}
\dd_\bx F \\ \dd_t F \end{array} \right)
$$
is a normal vector to the level of $\br_M$ through the point $(\bx,t)$. We see that 
$V \circ \gm \rightarrow \left( \begin{array}{c} 0  \\ \pm 1 \end{array} \right)$,
since 
$$(|\bx|\cdot \left|\frac{\dd_\bx F}{\nabla F}\right|) \circ\gm \to 0 \, , \; 
\left|\frac{\dd_\bx F}{\nabla F}\right| \circ\gm \to 0 \; \mbox{ and } \;
\frac{\dd_t F}{\vert \nabla F \vert} \to \pm 1.
$$ 
This contradicts the observation made just above.
\end{proof}

\bigskip
Since we just have seen that $t$-equisingularity at infinity implies the Malgrange condition, we need 
to check if there is a relation between these and sphericalness at infinity. 
To this end an obvious corollary of Proposition \ref{prop:spherical-exponent} is the following:
\begin{corollary}\label{cor:t-reg}
Let $c$ not be a generalized critical value. Then $\bt_M$ is horizontally spherical at $c$ at infinity. 
In other words $t$-equisingularity at infinity at $c$ implies horizontal sphericalness at $c$ at infinity.
\end{corollary}
\begin{proof}
It is just reformulating the fact that Malgrange at $c$ is equivalent to have 
$e_c \leq 0$ in Equation (\ref{eq:bochnak-Lojasiewicz-exponent}).
\end{proof}
We can now state the last result of this section about local triviality:
\begin{theorem}\label{thm:spherical-trivial}
Let $c$ be a value at which $\bt_M$ is horizontally spherical at infinity. Then $\bt_M$ is $C^{1+m}$ is 
locally trivial at $c$.
\end{theorem}
\begin{proof}
Once we have moved the origin of $\Rn\times 0$ so that its value is not $c$, we just have to integrate
the field $\chi = \frac{1}{|\nbt_M|}\nubtm$ as before. Inequality (\ref{eq:bochnak-Lojasiewicz-exponent})
now holds in $\bt_M^{-1}[c-\ve,c+\ve] \setminus \bB_{R_0}$ for a large positive $R_0$. As in \cite{DaGr1,DaGr2} combining 
it with Gronwall Lemma will show that any trajectory of $\chi$ parameterized over $[0,\ve]$ with initial point in $M_c\cap\bB_R$
stays in $\bB_{KR}$ for some constant $K$ depending only on $c$ and $\ve$.
\end{proof}
As a final remark, there are polynomial examples in \cite{DaGr1} with regular values which are ACV,
but with exponent $e_c < 1$.
%
%
%
%
%
%
%
%
%
%
%
%
%
%
%
%
%
%
%
%
%
%
%
%
%
%
%
%
%
%
%
%
%
\section{Curvature and absolute curvature of families of globally definable hypersurfaces}\label{section:curvature}
Some of the material presented here can also be found in \cite{Gra1} (or adapted from it).

\medskip
Let $H$ be a globally definable and oriented hypersurface of $\Rn$ of class $C^{1+m}$ with $m\geq 1$.

Assume that now $H$ is connected and let $\nu_H: H\mapsto \bS^{n-1}$ be an orientation. The unitary normal 
mapping $\nu_H$ is globally definable and $C^m$. 

Assume that the maximal rank of $\ud_\bx\nu_H$ when $\bx$ ranges $H$ is $n-1$. 

There exist finitely many definable disjoint connected open subsets $(U_i)_{i\in I}$
of $\bS^{n-1}$ such that 
$$
\clos(\nu_H(H)) = \cup_{i\in I} \clos(U_i)
$$
and for each $i\in I$, the mapping $\nu_H$ induces a globally definable finite covering 
$$
\nu_H :H_i \mapsto U_i
$$ 
where $H_i := \nu_H^{-1}(U_i)$ and such that 
$$
\dim \nu_H(H\setminus (\cup_{i\in I} H_i)) \leq n-2.
$$
Denoting $\kp_H$ the determinant of $\ud \nu_H$, that is the Gauss-Kronecker curvature of $M$ at the considered point,
the \em total Gauss curvature $K(H)$ of $H$ \em is defined (if it exists, and it does as we see below) as
$$
K(H) := \int_H \kp_H(\bx) \ud \bx.
$$
An application of the formula of change of variables gives
$$
\int_H \kp_H(\bx) \ud \bx = \sum_{i\in I} (-1)^{d_i} \vol_{n-1}(U_i)
$$
for $(-1)^{d_i}$ the degree of the covering mapping $\nu_H|_{H_i}:H_i\mapsto U_i$ for each $i$. 

\smallskip
We introduce another average of curvature, namely the \em total absolute curvature
$|K|(H)$ of $H$ \em defined as
$$
|K|(H):= \int_H |\kp_H(\bx)| \ud \bx
$$
Another application of the formula of change of variables yields, 
$$
\int_H |\kp_H(x)| \ud x = \sum_{i\in I} e_i \cdot\vol_{n-1}(U_i)
$$
where $e_i$ is the number of sheets of the covering $\nu_H|_{H_i}:H_i\mapsto U_i$. 

The hypothesis on the rank of $\ud \nu_H$ guarantees that $e_i$ is positive. Otherwise
both curvatures are $0$.

\bigskip
Returning to the notations and hypotheses of Section \ref{section:regularity}, 
the hypersurface $M$ can also be seen as a globally definable family of hypersurfaces $\cF_{\bt_M} := (T_c)_{c\in Im(\bt_M)}$ of
$\Rn$. We can define the following mapping:
$$
\begin{array}{rccl}
\bN : & M\setminus\crit(\bt_M) & \mapsto & \bS^{n-1} \\
 & (\bx,t) & \mapsto & \frac{\nu_M^\bx}{|\nu_M^\bx|} 
\end{array}.
$$
The mapping $N$ is called \em the Gauss mapping of the family $\cF_{\bt_M}$. \em
It is globally definable and $C^{1+m}$. The restriction of $N|_{T_c}$ is denoted $\bN_c$, so that
the family of mappings $(\bN_c)_{c\in Im(\bt_M)\setminus K_0(\bt_M)}$ is globally definable, where
$Im(\bt_M)$ is the image of the function $\bt_M$.
Let $\kp_c$ be the Gauss-Kronecker curvature of $T_c$.
Thus we can define two functions
$$
\begin{array}{rccl}
K : &  Im(\bt_M)\setminus K_0(\bt_M) & \to & \R \\
 & c & \mapsto & K(c) := \int_{T_c} \kp_c(\bp) \ud \bp, \\
|K| : &  Im(\bt_M)\setminus K_0(\bt_M) & \to & \R \\
 & c & \mapsto & |K|(c) := \int_{T_c} |\kp_c|(\bp) \ud \bp .
\end{array}
$$
The introductory material of this section guarantees that 
both functions are well defined. The paper \cite{Gra1} has dealt with the case
where $M$ is a graph. We can state now the result of this section:
\begin{theorem}\label{thm:continuity-curvatures}
(i) There are finitely many values in $Im(\bt_M) \setminus K_0(\bt_M)$ at which 
the function $t \mapsto K(t)$ is not continuous 

\smallskip\noindent
(ii) There are finitely many values in $Im(\bt_M) \setminus K_0(\bt_M)$ at which 
the function $t \mapsto |K|(t)$ is not continuous 

\smallskip\noindent
(iii) If $|K|$ is continuous at $c$, so is $K$.
\end{theorem}
\begin{proof}[Sketch of Proof]
It is a very similar proof to that of \cite[Sections 4,5,6]{Gra1}.
 
\smallskip
Let us consider the following globally definable and $C^{1+m}$ mapping 
$$
\begin{array}{rccl}
\Psi : & M & \mapsto & \bS^{n-1}\times \R \\
 & \bp & \mapsto & (\bN(\bp),\bt_M(\bp)).
\end{array}
$$ 
It is a local diffeomorphism at any point of $M \setminus \crit(\Psi)$.
Let $\Delta := \Psi(\crit(\Psi))$ which is definable, closed and of dimension lower than 
or equal to $n-1$. Let $\cU :=( \bS^{n-1}\times \R) \setminus \Delta$.
\\
There exists an integer number $p_M$ such that for any $(\bu,t)\in \cU$ the 
fibre $\Psi^{-1}(\bu,t)$ has at most $p_M$ points.
For any point $(\bu,t)$ in $\cU$ the degree $\delta (\bu,t)$ of $\Psi$
at $(\bu,t)$ may range from $-p_M$ to $p_M$. 
In particular the function $(\bu,t) \mapsto \delta(\bu,t)$ is definable and 
$$
\delta (\bu,t) = \deg_\bu \bN_t.
$$
We define the following subsets 
\begin{eqnarray*}
\cU_k & := & \{(\bu,t) \in \cU \, : \, \# \Psi^{-1}(\bu,t) = k  \} \\
U_t & := & \{\bu \in \bS^{n-1} \, : \, (\bu,t) \in\cU  \} = \bN_t(T_t\setminus\crit(\bN_t)) \\
U_{t,k} & := & \{\bu \in \bS^{n-1} \, : \, (\bu,t) \in\cU_k  \} .
\end{eqnarray*}
The subsets $U_t$ and $U_{t,k}$ are open, and we obtain finitely many globally definable
families $(U_t)_{t\in Im(\bt_M)\setminus K_0(\bt_M)}$ and 
$(U_{t,k})_{t\in Im(\bt_M)\setminus K_0(\bt_M)}$.

Note that $U_t = \cup_k U_{t,k}$ and since the function $\bu \mapsto \deg_\bu \bN_t$ is definable,
it is constant on each connected component of $U_t$. 

Let $c$ be a regular value of $\bt_M$. Since Hausdorff limits of closed definable
subsets of a given compact space exist, we can set
\begin{center}
\begin{tabular}{rcl}
$\scrV_c^+  :=  \lim_{t\to c,t>c} \clos (U_t)$ & and & 
$\scrV_{c,k}^+   :=  \lim_{t\to c,t>c} \clos (U_{t,k})$ \\
$\scrV_c^-  :=  \lim_{t\to c,t<c} \clos (U_t)$ & and & 
$\scrV_{c,k}^-  :=  \lim_{t\to c,t<c} \clos (U_{t,k}).$
\end{tabular}
\end{center}

\smallskip
Let $V_1,\ldots,V_{d_c}$ be the connected components of $U_c$. For each $i = 1,\ldots,s,$
let $k_i$ be the integer number such that $V_i \subset U_{c,k_i}$. 
For each $i=1,\ldots,d_c$ there exists $l_i^+ \geq k_i$ and $l_i^- \geq k_i$
such that 
$$
V_i \subset \scrV_{c,l_i^-}^- \; \mbox{ and } \; V_i \subset \scrV_{c,l_i^+}^+.
$$
In particular we deduce that for each $i$
$$
\vol_{n-1} (V_i) \leq \min \{ \, \vol_{n-1} (\scrV_{c,l_i^-}^-) \, , \,\vol_{n-1} (\scrV_{c,l_i^+}^+ ) \, \} \, .
$$
Let $\delta_i$ be the degree of $\bN_c$ at any point of $V_i$. We find
$$
K(c) = \sum_{i=1}^{d_c} \delta_i \cdot \vol_{n-1}(V_i) 
\; \mbox{ and } \;
|K|(s) = \sum_{i=1}^{d_c} k_i\cdot \vol_{n-1} (V_i).
$$
From the previous arguments we get that each following limit exists
$$
K_c^+ := \lim_{t\to c,t>c} K(t)\, , \;
K_c^- := \lim_{t\to c,t<c} K(t)\, , \;
|K|_c^+ := \lim_{t\to c,t>c} |K|(t)\, , \;
|K|_c^- := \lim_{t\to c,t<c} |K|(t)\, , \;
$$
and we obviously get 
$$
|K|(c) \leq \min (|K|_c^- \, , \, |K|_c^+ ).
$$
The rest of the proof follows from the following
arguments:
Assume that each $U_t$ has $d_t$ connected $V_{t,1},\ldots,V_{t,d_t}$.
Each such connected component $V_{t,i}$ lies in $U_{t,k_i(t)}$ 
with $k_i(t) \leq k_j(t)$ if and only if $i\leq j$. Moreover
the degree of $\bN_t$ at any point of $V_{t,i}$ is constant and equal to $\delta_i(t)$.
These comes from properties of $\Psi$ and $\cU$.
From here we deduce that there exists a finite subset $Z$ of $\R$ such that
for any $J$ connected component of $(\R\setminus Z)\cap Im(\bt_M)$, 
the numbers $d_t, k_i(t), \delta_i(t)$ are independent of $t$ in $J$.
Moreover each function $t \mapsto \vol_{n-1}(V_{t,i})$ is continuous
over $J$.
\end{proof}
%
%
%
%
%
%
%
%
%
%
%
%
%
%
%
%
%
%
%
%
%
%
%
%
%
%
%
%
%
%
%
%
\section{More on regularity at infinity}\label{section:more}

\medskip
Let $\bN : M\setminus \crit(\bt_M)\mapsto \bS^{n-1}$ be the Gauss mapping of the family of the regular levels of $\bt_M$.
Similarly to the conormal geometry at infinity (in $\Rn\times\R$) of the function $\bt_M$, we are interested in the 
limits of $\bN$ at infinity (in $\Rn$). 

\smallskip
Let $\Gm(\bN)$, contained in $M\times\bS^{n-1}$, be the graph of $\bN$, let $\overline{\Gm(\bN)}$ be its closure in 
$\Rnbar\times\R\times\bS^{n-1}$ and $\overline{\bN} : \overline{\Gm(\bN)} \mapsto\bS^{n-1}$ be the projection onto
$\bS^{n-1}$, so that we can think of it as the extension by continuity of $\bN$ to $\overline{\Gm(\bN)}$.

The closed definable subset $T_{c,+}^\infty$ is defined as 
$$
T_{c,+}^\infty := \{\bu\in\bS_\infty^{n-1} \, : \, \exists (\bp_k)_k \in M \, \mbox{ such that } \, \lim_\infty \bp_k = (\bu,c) \} .
$$

Let $V_c^\infty := \overline{\bN}(\pi^{-1}(T_{c,+}^\infty\times \{c\}))$, in other words it is the definable closed subset 
$$
V_c^\infty =\left\{\bv\in\bS^{n-1} \, : \, \exists ((\bx_k,\tau_k))_k \in M \hbox{ such that } \bx_k \to \infty, \, \tau_k \to c,
\, \bN (\bx_k,\tau_k)  \to \bv \right\},
$$
corresponding to all the limits at infinity of normals to the hypersurfaces $(T_t)_t$ as $t$ tends to $c$.

For each $\bu\in\bS_\infty^{n-1}$, let $V_{c,\bu}^\infty := V_c^\infty \cap \overline{\bN}(\pi^{-1}(\bu)\times \{c\})$, that
is 
$$
V_{c,\bu}^\infty =\left\{\bv\in\bS^{n-1} \, : \, \exists ((\bx_k,\tau_k))_k \in M \hbox{ such that }  \bx_k \to \infty, \, \tau_k \to c,
\, \frac{\bx_k}{|\bx_k|} \to\bu, \, \bN(\bx_k,\tau_k) \to\bv \right\}.
$$
Note that whenever $\bu$ does not belong to $T_{c,+}^\infty$ we find that $V_{c,\bu}^\infty$ is empty.

\smallskip
A very rigid consequence of $\bt_M$ being horizontally spherical at $c$ at infinity is the following:
\begin{lemma}\label{lem:normal-infty}
Let $c$ be a regular value taken by $\bt_M$ at which it is horizontally spherical at infinity. 
Then each $\bu$ in $T_{c,+}^\infty$ and each $\bv$ in $V_{c,\bu}^\infty$ are orthogonal.
\end{lemma}
\begin{proof}
Obvious from the definition of the horizontal sphericalness.
\end{proof}

\medskip
Let $c$ be a regular value taken by $\bt_M$ at which it is horizontally spherical at infinity.
Let $\ve$ be a positive real number such that for each $t \in [c-\ve,c+\ve]$ the function $\bt_M$ is 
horizontally spherical at $t$ at infinity. Let $T_{c,\ve} := \bt_M^{-1}([c+\ve,c-\ve])$.

We find that for each for $\eta$ in $]e_c,1[$, there exists a positive real number $R$ such that
for every $(\bx,t)$ belonging to $T_{c,\ve}\setminus \clos(\bB_R)$, we have 
\begin{equation}\label{eq:exponent}
|\bx|\cdot|\nbt_M (\bx,t)| \geq | t - c|^\eta .
\end{equation}

Let $\xi$ be the following definable vector field
$$
\xi := \frac{1}{|\nbt_M|}\frac{\nu_{\bt_M}^\bS}{|\nu_{\bt_M}^\bS|}, \; 
\mbox{ for } |(\bx,t)| \geq R \gg 1, \; \mbox{ and } |t-c|\leq \ve.
$$
It is definable and $C^{1+m}$, non vanishing, tangent to the Euclidean spheres. 
The flow of the differential equation 
$$
\dot{\bp}(t) = \xi(\bp(t)) \; \mbox{ and } \; \xi(0) \in T_c\times \{c\}\setminus \bB_R
$$
induces a $C^{1+m}$ diffeomorphism $(T_c\times\{c\}\setminus \bB_R)\times[-\ve,\ve] \mapsto T_{c,\ve}\setminus \bB_R$. 

\smallskip
Using Inequality (\ref{eq:exponent}) we deduce that the length $l(\bp_0,\bp_t)$ of the trajectory of 
$\xi$ between the point $\bp_0$ of $T_c\times\{c\}\setminus \bB_R$ and $\bp_t$, point reached after time $t$, is bounded as 
\begin{equation}\label{eq:gronwall}
l(\bp_0,\bp_t) \leq |\bp_0| \left(\frac{t^{1-\eta}}{1-\eta}\right) .
\end{equation}

Inequality (\ref{eq:gronwall}) implies that the angle $\tht(t)$ between the vector $\bp_t$ and 
$\bp_0$ tends to $0$ as $t$ goes to $0$. This proves the following:
\begin{lemma}\label{lem:small-tangent-cone}
Let $c$ be a regular value taken by the function $\bt_M$ at which it is horizontally spherical at infinity. Then 
$T_c^\infty = T_{c,+}^\infty$, thus $T_{c,+}^\infty$ is of dimension at most $n-2$.
\end{lemma}
%
%
%
%
%
%
%
%
%
%
%
%
%
%
%
%
%
%
%
%
%
%
%
%
%
%
%
%
%
%
%
\section{Main result}\label{section:main}
Our main result Theorem \ref{thm:main} presented in this section is a consequence of   results 
of equisingularity theory and of our context. 
\begin{theorem}\label{thm:main}
Let $F:\R^n\times\R\mapsto\R$ be a $C^{2+m}$ globally definable function over a polynomially bounded o-minimal structure, 
for a non negative integer number $m$. Assuming that $0$ is regular value of $F$, let $M$ be the level $\{F = 0\}$. Let $\bt_M$ be the 
projection of $M$ onto $\R$.

\smallskip
Let $c$ be a regular value taken by $\bt_M$ at which it is horizontally spherical at infinity. Then the total absolute curvature function 
$t \mapsto |K|(t)$ is continuous at $c$. Consequently the total curvature function $t \mapsto K(t)$ is continuous at $c$.
\end{theorem}
It is a straightforward consequence of the following 
\begin{lemma}\label{lem:dim-n-2}
Under the hypotheses of Theorem \ref{thm:main}, we find 
$$
\dim V_c^\infty \leq n-2.
$$
\end{lemma}
Let us show the main result.
\begin{proof}[Proof of the main result]
Let $|K| : t \mapsto |K|(t)$ be the total absolute curvature function of the family of hypersurfaces $(T_t)_t$.
By Lemma \ref{lem:dim-n-2} we find that  $V_c^\infty$  has $(n-1)$-dimensional volume zero.
Following \cite[Proposition 6.8]{Gra1}, we deduce there is no accumulation of curvature at infinity at $c$. 
In other words the function $|K|$ is continuous at $c$, and so is $K$ by point (iii) of Theorem 
\ref{thm:continuity-curvatures}. 
\end{proof}
Before going into the proof of Lemma \ref{lem:dim-n-2}, we observe that it states 
that there is no accumulation of curvature at infinity nearby the level $c$, or equivalently
there are no half-branch at infinity of the generic polar curve along which the function 
$\bt_M$ tends to $c$ (see \cite{Tib1,Gra2} for local triviality results with a similar flavor).
\begin{proof}[Proof of Lemma \ref{lem:dim-n-2}]
Let $\bv$ be a limit of normal direction lying in $V_c^\infty$. 
By the Curve Selection Lemma we can find a globally definable continuous path, going to infinity, 
along which this limit is reached: there exists such a path $\gm$ such that 
$$
\bN \circ \gm \to \bv \; \mbox{ and } \; \bt_M\circ\gm \to c.
$$
In particular there exists a positive exponent $\alpha$, in the field of exponents $\F_\cM$ of the structure $\cM$
such that 
$$
\frac{\bt_M\circ\gm - c}{|\gm|^{\alpha}} \to a \in \R^*.
$$
In other words there exists a positive exponent $e$ such that the germ at infinity of $\gm$ 
lies in 
$$
\cH_e := \{\bp \in T_{c,\ve}\setminus \bB_R \, : \, |\bt_M(\bp) - c| \leq |\bp|^{-e}\}.
$$
If the exponent $e$ belongs to $F_\cM$, then $\cH_e$ is globally definable and so is its closure $\overline{\cH_e}$ 
in $\Rnbar\times \R$. 
Let us define
$$
V_c^{\infty,e} := \left\{\bv\in \bS^{n-1} \, : \, \exists \; (\bp_k)_k \, \in \, 
\cH_e \, \hbox{ such that } \, \bN(\bp_k) \to \bv \right\} 
$$
which is a closed definable subset of $\bS^{n-1}$ contained in $V_c^\infty$ whenever $e$ lies in $\F_\cM$.

\smallskip
Let $\cH_e^\infty$ be the intersection $\overline{\cH_e} \cap \bS_\infty^{n-1} \times \{c\}$.
The function $\bt_M$ extends continuously and definably to $\overline{\cH_e}$ taking the value 
$c$ along $\cH_e^\infty$. Let $\bt_e$ be the restriction of this extension to $\overline{\cH_e}
\setminus \bB_R$.

According to \cite{Bekka,Loi}, we can stratify the pair $(\bt_e,\overline{\cH_e} \setminus \bB_R)$ 
with Thom's condition. Furthermore we can require that $X:=\cH_e\setminus \bB_R$ and $Y:=\cH_e^\infty$ 
are union of strata.

\smallskip
Suppose first that $X$ and $Y$ are strata. The dimension of $Y$ is $d \leq n-2$ since $\cH_e^\infty$ is contained 
in $T_{c,+}^\infty$, thus of dimension lower than or equal to $n-2$ by Lemma \ref{lem:small-tangent-cone}.
Let $\bp =(\bu,c)$ be a point of $Y$ and let $T := T_\bp Y$ which is contained in $\Rn\times 0$. Note that $T$ and $\bu$
are orthogonal. 

Let $\bv$ be a limit of the normal $\bN$
at infinity at $\bu$ taken into $\cH_e$ along a path $\gm$. We will show that $\bv$ and $T \oplus \R \bu$
are orthogonal. 
We recall that $\nu_M = \nu_M^\bx + \nu_M^t \dt$. Let $\nu$ be the limit of $\nu_M$ along $\gm$ as $\gm$ 
goes to infinity  and let $\eta$ be the limit of $\nubtm$. 
Writing $\nu$ as $(\nu^\bx,\nu^t)$ in $\Rn\times\R$, we have 
$$
\bv = \frac{\nu^\bx}{|\nu^\bx|} \; \mbox{ and }
\; \eta =-\nu^t \bv + \vert \nu^\bx \vert^2 \dt. 
$$
Thom's condition implies that $\eta$ and $T$ are orthogonal. Moreover, by horizontal spherical-ness at infinity, 
$\eta$ and $\bu$ are also orthogonal, therefore $\eta$ and $T \oplus \R \bu$ are orthogonal too. Hence, if 
$\nu^t \not=0$, then $\bv$ is orthogonal to $T \oplus \R \bu$ since $T \oplus \R \bu$ is contained in $\R^n \times 0$. 
If $\nu^t =0$ then $\bv = \nu$. Using the arguments of the proof of Lemma \ref{lem:transverse-horizontal-sphere}, 
we see that $\bu$ and $\nu$ are orthogonal. By Whitney's condition $(a)$, we know that $T_\bp Y$ is a subspace of
$\lim_\infty T_\gm M$ and so $T_\bp Y$ and $\nu$ are orthogonal. Hence we conclude that $\bv=\nu$ is orthogonal 
to $T_\bp Y \oplus \R \bu$.

\smallskip
Let $V_{c,\bu}^{\infty,e} :=  V_c^{\infty,e} \cap V_{c,\bu}^\infty$. 
We have proved that $\dim V_{c,\bu}^{\infty,e} \leq (n-1)-(d+1) = n-d-2$, and thus $\dim V_c^{\infty,e}\leq n-2$. 

\smallskip
In the general case the only thing to check is that whenever $X$ contains a (globally definable) stratum $S$ 
of dimension $s$ at most $n-1$, then its contribution to $V_c^\infty$ is at most of dimension $n-2$.
But this is so since the graph of $\bN|S$ is of dimension $s$, so that its limits at infinity
$$
\left\{\bv\in \bS^{n-1} \, : \, \exists S \ni (\bx_k,\tau_k)_k \, \hbox{ such that } \, 
\tau_k \to c\, , \,  \bN(\bx_k,\tau_k) \to \bv \right\} \subset V_c^\infty
$$
have dimension at most $s-1 \leq n-2$.

\smallskip
We conclude that $V_c^{\infty,e}$ has dimension lower than or equal to $n-2$ for any exponent $e$ of $\F_\cM$.

\medskip
Since any limit $\bv$ of $V_c^\infty$ belongs to some $V_c^{\infty,e}$ for some $e$ in $\F_\cM$, and since the
family $(V_c^{\infty,e})_{e\in(\F_\cM){>0}}$ is increasing as $e$ goes to $0$, we get that $V_c^\infty$ is the Hausdorff 
limit at $e=0$ of $V_c^{\infty,e}$, thus has dimension lower than or equal to $n-2$.
\end{proof}
We conclude with an interesting observation. For this purpose we need a few more preparations. 
Let $c$ be regular value taken by $\bt_M$. Let $\ve$ be a positive number such that $[c-\ve,c+\ve]$ 
consists only of regular values. Let $Z$ be a connected component of $\bt_M^{-1}(]c-\ve,c+\ve[)$. 
Let us consider now $\bt_Z$ the restriction of $\bt_M$ to $Z$. 
Let $Z_t \times \{t\} := \bt_Z^{-1}(t) = M_t \cap Z$.
Let $K_Z(t) := \int_{Z_t} \kp$ and $|K|_Z(t) := \int_{Z_t} |\kp|$ for $t$ in $]c-\ve,c+\ve[$. 
Then we actually have showed the following:
\begin{theorem}\label{thm:main-connected}
Under the above hypotheses, assume furthermore that $\bt_Z$ is horizontally spherical at infinity at $c$. Then
the functions $K_Z$ and $|K|_Z$ are continuous at $c$.
\end{theorem}
To rephrase informally Theorem \ref{thm:main-connected}, the continuity of $t\mapsto |K|(t)$ nearby 
the value $c$ at which the function $\bt_M$ is horizontally spherical at infinity, is equivalent to
the continuity nearby $c$
of each function $t\mapsto |K|_Z(t)$ for each connected component $Z$ of $\bt_M^{-1}(]c-\ve,c+\ve[)$.
%
%
%
%
%
%
%
%
%
%
%
%
%
%
%
%
%
%
%
%
%
%
%
%
%
%
%
%
%
%
%
%
%
%
%
%
%

%
%
%
%
%
%
%
%
%
%
%
%
%
%
%
%
%
%
%
%
%
%
%
%
\section{The special case of functions}
We treat here briefly the case of functions which is a special case of the context presented here.
The continuity of curvatures is the same property but the regularity conditions are a little 
bit different. 

\medskip
Let $f:\R^n \mapsto \R$ be a $C^{2+m}$, with non-negative $m$, globally definable function. 
We denote the level $f^{-1}(t)$ by $F_t$ and its closure in the spherical compactification 
by $\overline{F_t}$. Its intersection with the sphere at infinity $\bS_\infty^{n-1}$ will be denoted 
$F_t^\infty$. 
Let $\nu_f$ be the unitary gradient field $\frac{\nabla f}{|\nabla f|}$.

\smallskip
The function $f$ satisfies \em Malgrange condition \em at $c$ if there are positive constants 
$R,\ve, A$ such that 
$$
|\bx|> R, \, |f(\bx) - c|< \ve \; \Longrightarrow \; |\bx|\cdot|\nabla f(\bx) | \geq A|f(\bx) - c|.
$$

\smallskip
We would like to introduce what the analogue of horizontal spherical-ness in this context would be.
The function $f$ is \em spherical at the regular value $c$ at infinity \em if 
along any sequence of points $(\bx_k)_k$ of $\R^n$ such that 
$|\bx_k|$ goes to $\infty$ and $f(\bx_k)$ goes to $c$, we have
$$
\left\langle \lim_\infty \nu_f (\bx_k) \; , \; \lim_\infty \frac{\bx_k}{|\bx_k|} \right\rangle = 0,
$$
whenever each limit exists.
\\
This condition is equivalent to the following result already proved in \cite{DaGr1,DaGr2}
which justified the introduction for families of the notion of horizontal spherical-ness at infinity.
\begin{theorem}[\cite{DaGr1,DaGr2}]
Let $c$ be a regular value of $f$ taken by $f$. The function $f$ is spherical at infinity at $c$
if and only if there exists an exponent $e_c$ in $\F_\cM \cap ]-\infty,1[$ and a positive constant 
$E_c$ such that
$$
|\bx| \gg 1, \, |f(\bx - c|\ll 1 \; \Longrightarrow \; |\bx|\cdot|\nabla f(\bx)| \geq E_c |f(\bx) - c|^{e_c}.
$$
\end{theorem}
It is well known that $t$-regularity is equivalent to Malgrange \cite{DRT} (their proof goes through 
the globally definable context) and that Malgrange is equivalent  to requiring having $e_c \leq 0$, thus
spherical-ness at infinity. 

\medskip
Let $K(t)$ be the total Gauss-Kronecker curvature of $F_t$ and $|K|(t)$ be the total absolute Gauss-Kronecker
curvature of $F_t$. In the context of functions what we have proved is the following:
\begin{theorem}\label{thm:main-function}
Let $f:\R^n \mapsto \R$ be a globally definable $C^{2+m}$ function for some non-negative integer $m$.
Let $c$ be a regular value at which the function is spherical at infinity. 

(1) Then the function $t \mapsto |K|(t)|$ is continuous at $c$, and thus so is $t\mapsto K(t)$.

(2) As for Theorem \ref{thm:main-connected}, for any connected component $Z$ of $f^{-1}]c-\ve,c+\ve[$
for positive $\ve$ small enough, the function $t\mapsto |K|_Z(t)$ is continuous at $c$, and thus so
is $t\mapsto K_Z(t)$.
\end{theorem}

Let us end with a last result on equisingularity of the family of fibres of a function.
\begin{corollary}
Let $f:\R^n \mapsto \R$ be a globally definable $C^{2+m}$ function for some non-negative integer $m$.
Let $c$ be a regular value at which the function is spherical at infinity. 

If $n$ is odd then the following function is continuous at $c$
$$
t \mapsto \int_{\bG (n-1,n)} \chi \left( f^{-1}(t) \cap H \right) dH .
$$

If $n$ is even then the following function is continuous at $c$
$$
t \mapsto \int_{\bG (n-1,n)} \left[\chi \left( \{f \ge t\} \cap H \right) - \chi \left( \{f \le t\} 
\cap H \right)  \right] dH .
$$
\end{corollary}
\proof By Theorem \ref{thm:main-function}, we know that the function $t\mapsto K(t)$ is continuous at $c$. 
Then we apply Theorem 4.5 in \cite{DutertreGeoDedicata2004}. If $n$ is odd, the result is clear because the 
function $t \mapsto \chi \left( f^{-1} (t) \right)$ is constant in a neighborhood of $c$. If $n$ is even, it 
is enough to prove that the functions $t \mapsto \chi \left( \{f \ge t \} \right)$ and 
$t \mapsto \chi \left( \{f \le t \} \right)$ are constant in a neighborhood of $c$. By the Mayer-Vietoris 
sequence, if $t>c$  then we have 
$$
\chi \left( \{f \ge c \} \right) = \chi \left( \{f \ge t \} \right) + \chi 
\left( \{ c \le f \le t \} \right) - \chi \left( f^{-1}(t)  \right).
$$
So if $t$ is close enough to $c$ then $\chi \left( \{f \ge c \} \right) = \chi \left( \{f \ge t \} \right)$, 
for $f$ is a fibration over $[c,t]$. Similarly we can show that $\chi \left( \{f \le c \} \right) = 
\chi \left( \{f \le t \} \right)$ for $t>c$ close enough to $c$. The same argument works for $t<c$. 
\endproof
%
%
%
%
%
%
%
%
%
%
%
%
%
%
%
%
%
%
%
%
%
%
%
%
%
%
%
%
%

%
%

%
%
%
%
%
%
%
%
%
%
%
%
%
%
%
%
%
%
%
%
%
%
%
%
%
%

%
%

%
%
%
%
%
%
%
%
%
%
%
%
%
%
%
%
%
%
%
%
%
%
%
%
%
%
%
%
%
%

\end{document}